\documentclass[12pt, twoside, leqno]{article}
\usepackage{latexsym}
\usepackage{amsmath}
\usepackage{amsfonts}
\usepackage{amssymb}
\usepackage{amsthm}
\usepackage{color}
\numberwithin{equation}{section}

\input xy
\xyoption{all} \hfuzz=1.5pt \tolerance=500
\theoremstyle{plain}

\newtheorem{theorem}{\indent Theorem}[section]
\newtheorem{proposition}[theorem]{\indent Proposition}

\theoremstyle{definition}
\newtheorem{definition}[theorem]{\indent Definition}
\newtheorem{example}[theorem]{\indent Example}
\newtheorem{remark}[theorem]{\indent Remark}

\newcommand{\cd}{{\cdot}}
\newcommand{\ot}{{\otimes}}
\newcommand{\hil}{\stackrel{\cdot}{\otimes}}
\newcommand{\cK}{{\cal F}}
\newcommand{\la}{\langle}
\newcommand{\ra}{\rangle}
\newcommand{\id}{{\bf 1}}
\newcommand{\di}{\diamondsuit}
\newcommand{\q}{\quad}
\newcommand{\qq}{\qquad}
\newcommand{\va}{\varphi}
\newcommand{\rr}{{\cal R}}
\newcommand{\ii}{\infty}
\newcommand{\mt}{\mapsto}
\newcommand{\co}{{\mathbb C}}
\newcommand{\mmdd}{\mathrel{\mathop{\otimes}\limits_{\cdot}}}

\newcommand{\bb}{{\cal B}}

\newcommand{\al}{{\alpha}}     
     
\newcommand{\de}{{\delta}}     \newcommand{\lm}{{\lambda}}
    
\newcommand{\om} {{\omega}}

\def\x{\relax\ifmmode {\mbox{*}}\else*\fi}

\newcommand{\ed}{\end{document}}

\begin{document}
\pagestyle{myheadings} \markboth{\centerline{\small{\sc
            A.~Ya.~Helemskii}}}
         {\centerline{\small{\sc Structures on the way  from classical to quantum spaces}}}
 \title{\bf Structures on the way  from classical to quantum spaces and their tensor products }

\author{\bf A.~Ya.~Helemskii}

\date{}

\maketitle

\noindent
\bigskip
\centerline{{\it In memoriam: Professor Charles Read}}

\bigskip
\begin{abstract}
We study tensor products of two structures situated, in a sense, between normed spaces and 
(abstract) operator spaces. We call them Lambert and proto-Lambert spaces and pay more attention to 
the latter ones. 
The considered two tensor products 
lead to essentially different norms in the respective spaces. Moreover, the proto-Lambert tensor 
product is especially nice for spaces with the maximal proto-Lambert norm and in particular, for 
$L_1$-spaces. At the same time the Lambert tensor product is nice for Hilbert spaces with the 
minimal Lambert norm. 
\end{abstract}

\setcounter{equation} {0}   
\section{Introduction}

The subject of the present paper is a structure on a linear space that, in a reasonable sense, is 
situated between the classical structure of a normed space and the structure of an abstract 
operator, or quantum space. The latter structure was discovered about 35 years ago; nowadays the 
theory of operator spaces, sometimes called quantum functional analysis, is a well developed area 
of modern functional analysis, presented in widely known textbooks~\cite{efr,pis,paul,blem}. 
Leading idea of that area was to investigate not just norm on a given linear space, say $E$, but a 
sequence of norms $\|\cd\|_n;n=1,2,\dots$, each one on the space of $n\times n$ matrices with 
entries in $E$, mutually related by certain natural conditions, the so-called Ruan axioms. 

 The above-mentioned intermediate structure appeared in 2002 in the Ph.D thesis of
A. Lambert~\cite{lam}; his superviser was G.Wittstock, one of the founding fathers of operator 
space theory. It was Lambert who suggested to consider, for every $n$, not a norm 
 on the matrix space $M_n(E)$ but a norm on the column space of length $n$,
consisting of vectors from $E$. He called the resulting sequence of norms {\it operator-sequential 
norm on $E$}, if it satisfied two natural axioms. Lambert developed a beautiful and rich theory, in 
particular, clarifying (putting in proper perspective) some aspects of quantum as well as classical 
functional analysis. As one of the achievements of his theory, Lambert shows that for his spaces 
there exists a concept of tensor product with good properties. One can say that his tensor product 
is on the way from the projective tensor product of normed spaces to the operator-projective tensor 
product of quantum spaces. 

In the present paper we study some properties of the Lambert's tensor product. But we also pay much 
attention to a certain natural generalization of 
Lambert's ``operator--sequential space''. It is called a proto-Lambert space, and it arises when we 
assume that only the first of Lambert's axioms for his spaces  is fulfilled. 
Our point is that one can obtain 
 many good things, if he considers proto--Lambert, and not, generally speaking, Lambert spaces. 
 
  Note that these
 proto--Lambert spaces are, after translation into an equivalent language, an important
 particular case of the so--called $p$--multi--normed spaces. The latter were quite recently (just in 
 time when this paper was under preparation)
 introduced and successfully 
 studied by  H.Dales, N.Laustsen, T.Oikhberg and V.Troitsky in the memoir~\cite{dal}. Proto--Lambert
 spaces correspond to the case of $p=2$. Thus, 
 they share the general properties of $p$--multi--normed spaces, investigated in the cited memoir. However,
 the properties, considered in our paper, rely heavily on the specific advantages of that particular 
 $p$, actually the same advantages that distinguish $\ell_2$ among all $\ell_p$.

Our presentation will be given in the frame--work of the so--called non--coordinate 
(``index--free'') approach to the structures in question, similar to what was done in~\cite{heb2} 
for operator spaces. Thus, it is different from the original approach in~\cite{lam}. Both ways of 
presentation have their own advantages (and drawbacks), but in questions, revolving around tensor 
products, the non--coordinate, ``index--free'' presentation is, in our subjective opinion, more 
elegant and transparent. 

The contents of the paper are as follows. 

 The second section contains the definition of a proto--Lambert (still not Lambert) space
 and some examples, notably the space $L_p(X,E);1\le p<\ii$ of relevant $E$--valued measurable 
 functions on $X$.
 (Running ahead, we note that this space is a Lambert space only when $p\ge2$).

 In Section 3, we consider classes of maps that reflect in a proper way the structure of a
 proto--Lambert space: $L$--bounded and $L$--contractive linear and bilinear operators. We prove that
 some classes of (bi)linear operators have the relevant properties and, in particular,
  some bilinear operators, related to $L_p$--spaces and to ``classical'' projective tensor products of
  normed spaces, are completely contractive.

 In Section 4, we show that proto--Lambert spaces have their own tensor product \\ `` $\ot_{pl}$ '',
  possessing the universal property for the class of $L$--bounded  bilinear operators between
these spaces. 

In Section 5, we concentrate on the case when one of the tensor factors is a space with the 
so--called maximal proto-Lambert norm, in particular, an $L_1(X)$--space, and give an explicit 
description of the resulting  proto-Lambert tensor product. As a corollary, we obtain a version, 
for proto-Lambert spaces, of the Grothendieck's theorem on tensoring by $L_1$--spaces in the 
``classical'' context of Banach spaces: cf., e.g.,~\cite[\S2,n$^o$2]{gro}. 


In Section 6, we pass from proto--Lambert to Lambert spaces, adding in the relevant definition the 
non-coordinate analogue of the second of the Lambert's axioms. We introduce the respective version 
`` $\ot_{l}$ ''of Lambert's ``maximal tensor product'' of his Operatorfolgenr\"aume and prove its 
existence. 

We have seen in Section 5 that the proto-Lambert tensor product is especially nice for 
$L_1$--spaces. In Section 7 we show that the Lambert (without ``proto--'') tensor product is nice 
for Hilbert spaces. Namely, if we equip both of Hilbert spaces with the so--called minimal Lambert 
norm, then their completed Lambert tensor product is again a Hilbert space with the same structure. 

In the last Section 8, we compare both tensor products,`` $\ot_{pl}$ '' and `` $\ot_{l}$ '', and 
show that the first one provides, generally speaking, essentially greater norms. In particular, for 
every $n$ we display a certain element in the amplification of the tensor square of a certain 
Lambert space. It turns out that the norm of this element,  provided by the proto--Lambert tensor 
product, is $n$, whereas the norm, provided by the Lambert tensor product, is $\sqrt{n}$. 

\setcounter{equation} {1} 
\section{Proto-Lambert spaces and their examples}
To begin with, we choose an arbitrary, separable, infinite-dimensional Hilbert space, denote it by 
$H$ and fix it throughout the whole paper. The identity operator on $H$ will be denoted by $\id$. 

As usual, by $\bb(E,F)$ we denote the space of all bounded operators between the normed spaces $E$ 
and $F$, endowed with the operator norm. We write $\bb(E)$ instead of $\bb(E,E)$, and also $\bb$ 
instead of $\bb(H)$. 

If $K,L$ are pre--Hilbert spaces, 
 $x\in K, y\in L$, we denote by $x\circ y:L\to K$ the rank 1 operator, taking $z$ to $\la z,y\ra x$ . Note that we have $\|x\circ y\|=\|x\|\|y\|$.

The symbol $\ot$ is used for the (algebraic) tensor product of linear spaces and for 
 elementary tensors. The symbols $\ot_p$ and $\ot_i$ denote the non--completed projective and injective 
 tensor product of normed spaces, respectively, and the symbol $\ot_{hil}$ is used for the non-completed 
 Hilbert tensor product of pre--Hilbert spaces. The symbols $\widehat\ot_p$, $\widehat\ot_i$ and 
 $\widehat\ot_{hil}$ are used for the respective completed tensor products.
The complex--conjugate space of a linear space $E$ is denoted by $E^{cc}$. 

\medskip
In what follows we need the triple notion of the so-called amplification. First, we amplify linear 
spaces, then linear operators and finally bilinear operators. Note that these amplifications differ 
from the amplifications, serving in the theory of quantum spaces (cf.~\cite{heb2}). 

The {\it amplification} of a given linear space $E$ is the tensor product $H\ot E$. Usually we 
briefly denote it by $HE$, and an elementary tensor, say $\xi\ot x; \xi\in H, x\in E$, by $\xi x$. 
Note that $HE$ is a left module over the algebra $\bb$ with the outer multiplication `` $\cd$ '', 
well defined by $a\cd(\xi x):=a(\xi)x$. 
\begin{remark}
In the non-coordinate presentation of the operator space theory the amplification of $E$ is $\cK\ot
E$, where $\cK$ is the space of finite rank bounded operators on $H$. But $\cK=H\ot H^{cc}$; so,
passing to the ``non-coordinate Lambert theory'', we replace the whole tensor product by its first
factor. (One can observe a similar transfer in the coordinate presentation: we replace
$M_n(E)=\co^n\ot(\co^n)^{cc}$ by $\co^n$).

Note that the transfer from $\cK\ot E$ to $H\ot E$ was actually used in the ``non-coordinate'' 
proof of the injective property of the Haagerup tensor product of operator spaces~\cite[Section 
7.3]{heb2} (such a property was discovered by Paulsen/Smith~\cite{pas}). 
\end{remark}
\begin{definition} \label{def1} 
A semi-norm on $HE$ is called {\it proto-Lambert semi-norm} or briefly {\it $PL$-semi-norm} on $E$, 
if the left $\bb$-module $HE$ is contractive, that is we always have the estimate $\|a\cd 
u\|\le\|a\|\|u\|$. The space $E$, endowed by a $PL$-semi-norm, is called {\it semi-normed 
proto-Lambert space} or briefly {\it semi-normed $PL$-space}). If the semi-norm in question is 
actually a norm, we speak, naturally, about a {\it normed proto-Lambert} ( = {\it normed $PL$--) 
space}, and in this case we often omit the word ``normed''. 
\end{definition}
A semi-normed $PL$--space $E$ becomes semi-normed space (in the usual sense), if  for $x\in E$ we 
set $\|x\|:=\|\xi x\|$, where $\xi\in H$ is an arbitrary vector with $\|\xi\|=1$. Clearly, the 
result does not depend on a choice of $\xi$. The obtained semi-normed space is called {\it 
underlying space} of a given $PL$-space, and the latter is called a {\it $PL$--quantization} of a 
former. (We use such a term by analogy with quantizations in operator space theory; see, 
e.g.,~\cite{ef5},~\cite{efr} or~\cite{heb2}). Obviously, for all $\xi\in H$ and $x\in E$ we have 
$\|\xi x\|=\|\xi\|\|x\|$. 

\medskip
It is easy to verify that the space of scalars, $\co$, has the only $PL$--quantization, given by 
the identification $H\co=H$. 
\begin{proposition} \label{pro1} 
 Let $E$ be a semi-normed $PL$--space with a normed underlying space. Then the $PL$--semi-norm on 
 $HE$ is a norm.
\end{proposition}
\begin{proof}
Take $u\in HE; u\ne0$ and represent it as $\sum_{k=1}^n\xi_kx_k$, where $\xi_k$ are linearly 
independent, $\|\xi_1\|=1$ and $x_1\ne0$. Further, take $\eta\in H$ with $\la\xi_1,\eta\ra=1$ and 
$\la\xi_k,\eta\ra=0$ for $k>1$. Then $(\xi_1\circ\eta)\cd u=\xi_1x_1$. Therefore we have 
$0<\|\xi_1x_1\|\le\|\xi_1\circ\eta\|\|u\|=\|\xi_1\|\|\eta\|\|u\|$, hence $\|u\|>0$. 
\end{proof}
\begin{example} \label{ex1}
 Every normed space, say $E$, has, generally speaking, a lot of $PL$--quantizations. We distinguish two 
 of them. The 
 $PL$--space, denoted by $E_{\max}$, respectively $E_{\min}$, has the $PL$--norm, obtained by the endowing $HE$ with the norm of $H\ot_p E$,
respectively of $H\ot_i E$. We denote the norm on the former and on the latter space by 
$\|\cd\|_{\max}$ and $\|\cd\|_{\min}$, respectively; accordingly, the corresponding 
$PL$-quantizations of $E$ will be called maximal and minimal. Clearly, the $PL$-norm of $E_{\max}$ 
is the greatest of all $PL$-norms of $PL$-quantizations of $E$. The adjective ``minimal'' will be 
justified a little bit later. 
\end{example}
\begin{example} \label{ex3} 
   Let $(X,\mu)$ be a measure space, $L_p(X);1\le p<\ii$ the relevant Banach space, $F$ a normed 
   $PL$-space (say, $F:=\co$ in the simplest case). We want to endow the ``classical'' space $L_p(X,F)$ 
   (of relevant $F$-valued functions) with a $PL$-norm.

As a preliminary step, consider  the normed space $L_p(X,HF)$ of relevant $HF$-valued measurable 
functions on $X$ and observe that it is a left $\bb$-module with the outer multiplication defined 
by $[a\cd\bar x](t):=a\cd[\bar x(t)]; a\in\bb,\bar x\in L_p(X,HF), t\in X$. A routine calculation 
shows that this module is contractive. 

Now consider the operator $\al:H(L_p(X,F))\to L_p(X,HF)$, well defined on elementary tensors by 
taking $\xi x; x\in L_p(X,F),\xi\in H$ to the $HF$-valued function $\bar x(t):=\xi(x(t))$, and 
introduce  the semi-norm on $H(L_p(X,F))$, setting $\|u\|:=\|\al(u)\|$. It is easy to veryfy that
$\al$ is a $\bb$-module morphism.
 Thus there is an isometric morphism of the module $H(L_p(X,F))$ into a contractive module. It follows
 immediately that the former module is itself contractive, that is the introduced semi-norm on 
 $H(L_p(X,F))$ is a $PL$-semi-norm on $L_p(X,F)$.
Further, for $\xi\in H; \|\xi\|=1$ and $x\in L_p(X,F)$ we have $\|\xi[x(t)]\|=\|x(t)\|$ for all 
$t\in X$. Therefore for $\xi x\in H(L_p(X,F))$ we have 
$$
\|\xi x\|=\left(\int_X\|[\xi x](t)\|^pd\mu(t)\right)^\frac{1}{p}=
\left(\int_X\|x(t)\|^pd\mu(t)\right)^\frac{1}{p}.
$$
We see that the underlying semi-normed space of the constructed $PL$-space is $L_p(X,F)$. 
 Therefore Proposition \ref{pro1} guarantees that the introduced $PL$-seminorm on $L_p(X,F)$ is 
 actually a norm.
\end{example}
\begin{example} \label {ex8}
We want to introduce a $PL$--quantization of the ``classical'' tensor product $E\ot_p F$ of normed 
spaces, when one of tensor factors, say, to be definite, $F$, is a $PL$--space. 

Consider the linear isomorphism $\beta: H(E\ot F)\to E\ot_p(HF):\xi(x\ot y)\mt x\ot\xi y$ and 
introduce a norm on $H(E\ot F)$ by setting $\|U\|:=\|\beta(U)\|$. The space $E\ot_p(HF)$, as a 
projective tensor product of a normed space and a contractive $\bb$-module, has itself a standard 
structure of a contractive $\bb$-module. Since $\beta$ is a $\bb$-module morphism, the same is true 
with $H(E\ot F)$. Thus $E\ot F$ becomes a $PL$--space, and we must show that its underlying normed 
space $(E\ot F,\|\cd\|$ is exactly $E\ot_p F$. 

Denote the norm on $E\ot_p F$ and on $E\ot_p(HF)$ by $\|\cd\|_p$. Take arbitrary $u\in E\ot F$.
It is easy to check that the norm $\|\cd\|$ on $E\ot F$ is a cross-norm, we have $\|\xi 
u\|\le\|u\|_p$. Therefore our task is to show that, for $\xi\in H,\|\xi\|=1$, we have $\|\xi 
u\|\ge\|u\|_p$. 

Identifying $\bb$-modules $H(E\ot F)$ and $E\ot_p(HF)$ by means of $\beta$, 
 represent $\xi u$ as $\sum_{k=1}^nx_k\ot w_k; x_k\in E,w_k\in HF$. Set $p:=\xi\circ\xi$. 
Obviously, $p\cd w_k=\xi y_k$ for some $y_k\in F; k=1,...,n$. Therefore 
$\sum_{k=1}^n\|x_k\|\|w_k\|\ge \sum_{k=1}^n\|x_k\|\|p\cd w_k\|= 
\sum_{k=1}^n\|x_k\|\|y_k\|$. But we have $\xi u=p\cd(\xi u)=\sum_{k=1}^nx_k\ot p\cd w_k= 
\xi(\sum_{k=1}^nx_k\ot y_k)$. It follows that $u=\sum_{k=1}^nx_k\ot y_k$. Consequently, 
$\sum_{k=1}^n\|x_k\|\|w_k\|\ge\|u\|_p$, and we are done. 

From now on we denote the constructed $PL$--quantization of $E\ot_p F$ again by $E\ot_p F$. 
\end{example}

\section{$L$-bounded linear and bilinear operators}
Suppose we are given an operator $\va:E\to F$ between linear spaces. Denote, for brevity, the 
operator $\id\ot\va:HE\to HF$ (taking $\xi x$ to $\xi\va(x)$) by $\va_\ii$ and call it {\it 
amplification} of $\va$. Obviously, $\va_\ii$ is a morphism of left $\bb$-modules. 
\begin{definition} \label{def2} 
An operator $\va:E\to F$ between seminormed $PL$--spaces is called {\it $L$--bounded}, {\it 
$L$-contractive}, {\it $L$--isometric, $L$--isometric isomorphism},  if $\va_\ii$ is bounded, 
contractive, isometric, isometric isomorphism, respectively. We set $\|\va\|_{lb}:=\|\va_\ii\|$. 
\end{definition}
If $\va$ is bounded, being considered between the respective underlying seminormed spaces, we say 
that it is (just) {\it bounded}, and denote its operator seminorm, as usual, by $\|\va\|$. Clearly, 
every $L$--bounded operator $\va:E\to F$ is bounded, and $\|\va\|\le\|\va\|_{lb}$ . 

\medskip
Some operators between $PL$--spaces, bounded as operators between underlying spaces, are 
``automatically'' $L$--bounded. Here is the first phenomenon of that kind. 

\begin{proposition} \label{autc}
    Let $E$ be a $PL$--space. Then every bounded functional $f:E\to\co$ is $L$--bounded, and $\|f\|_{lb}:=\|f\|$.
\end{proposition}
\begin{proof}
 Clearly, it is sufficient to show that for every $u\in HE$ we have $\|f_\ii(u)\|\le\|f\|\|u\|$.
Recall that $\|f_\ii(u)\|=\max\{|\la f_\ii(u),\xi\ra|; \xi\in H, \|\xi\|=1\}$. Presenting $u$ as a 
sum of elementary tensors, we see that, for every $\eta\in H;\|\eta\|=1$ we have $\la 
f_\ii(u),\xi\ra\eta=f_\ii[(\eta\circ\xi)\cd u]$ and also $(\eta\circ\xi)\cd u=\eta x_\xi$ for some 
$x_\xi\in E$. It follows that $\|x_\xi\|=\|(\eta\circ\xi)\cd u\|\le\|u\|$, hence $|\la 
f_\ii(u),\xi\ra|=\|f_\ii[(\eta\circ\xi)\cd u]\|=\|f_\ii(\eta x_\xi)\|=|f(x_\xi)|\le\|f\|\|u\|$. 
\end{proof}
Thus for every $PL$--space $E$ and $u\in HE$ we have $\|u\|\ge\sup\{\|f_\ii(u)\|\}$, where supremum 
is taken over all $f\in E^*; \|f\|\le1$. But such a supremum is exactly $\|u\|_{\min}$. This 
justifies  the word ``minimal'' in Example \ref{ex1}.

%
%
%
A $PL$--space is called {\it complete} (or Banach), if its underlying normed space is complete. As 
in the ``classical'' context, for every $PL$--space $E$ there exists its {\it completion}, which is 
defined as a pair $(\overline{E},i:E\to\overline{E})$, consisting of a complete $PL$--space and an 
$L$--isometric operator, such that the same pair, considered for respective underlying spaces and 
operators, is the ``classical'' completion of $E$ as of a normed space. The proof of the respective 
existence theorem repeats, with obvious modifications, the simple argument given in~\cite[Chapter 
4]{heb2} for quantum spaces. We only recall that the norm on $H\overline{E}$ is introduced with the 
help of the natural embedding of $H\overline{E}$ into $\overline{HE}$, the ``classical'' completion 
of $HE$. This embedding is well defined by taking an elementary tensor $\xi x;\xi\in H, 
x\in\overline{E}$ to $\lim_{n\to\ii}\xi x_n$, where $x_n\in E$ converges to $x$; hence $\xi x_n$ 
can be considered as a converging sequence in $\overline{HE}$. (Here, of course, we identify $E$ 
with a subspace of $\overline{E}$ and $HE$ with a subspace of $\overline{HE}$.) 

 It is easy to observe that the characteristic universal property of the ``classical'' completion has its proto--Lambert version. Namely, if $(\overline{E}, i)$ is the
  completion of a $PL$--space $E, \q F$ a $PL$--space and $\va: E\to F$ is an $L$--bounded operator, 
  then there exists a unique $L$--bounded operator
  $\overline{\va}:\overline{E}\to\overline{F}$, which is, in obvious sense, the {\it continuous extension} 
  of $\va$. Moreover, we have $\|\overline{\va}\|_{lb}=\|{\va}\|_{lb}$.

Distinguish the following useful fact. Its proof is the same, up to obvious modifications, as of 
Proposition 4.8 in~\cite{heb2}. 

\begin{proposition} \label{context}
Let $\va:E\to F$ be an $L$--isometric isomorhism between $PL$--spaces. Then its continuous 
extension $\overline{\va}:\overline{E}\to\overline{F}$ is also an $L$--isometric isomorhism. 
\end{proposition}

\medskip
We pass to bilinear operators. By virtue of Riesz/Fisher Theorem, we can arbitrarily choose a 
unitary isomorphism $\iota:H\widehat\ot_{hil} H\to H$ and fix it throughout the whole paper. 
Following~\cite{he2}, for $\xi,\eta\in H$ we denote the vector $\iota(\xi\ot\eta)\in H$ by 
$\xi\di\eta$, and for $a,b\in\bb$ we denote the operator $ \iota(a\hil b)\iota^{-1}$ on $H$ by 
$a\di b$; obviously, the latter is well defined by the equality $(a\di b)(\xi\di\eta)= a(\xi)\di 
b(\eta)$. Evidently, we have 
\begin{align} \label{gooddi}
\|\xi\di\eta\|=\|\xi\|\|\eta\| \qq{\rm and}\qq \|a\di b\|=\|a\|\|b\|.
\end{align}
If $E$ is a linear space, $\xi\in H$ and $u\in HE$, we set $\xi\di u:=T_\xi\cd u$, where  
$T_\xi\in\bb$ sends $\eta$ to $\xi\di\eta$. Thus, this version of the operation ` $\di$ ' is well 
defined on elementary tensors by $\xi\di\eta  x:=(\xi\di\eta)x$. Similarly, we introduce 
$u\di\eta\in HE$ by $\xi x\di\eta:=(\xi\di\eta)x$. By (\ref{gooddi}), $T_\xi=\|\xi\|S$, where $S$ 
is an isometry. Therefore, if $E$ is a $PL$--space, we have 
 \begin{align} \label{good2} 
\|\xi\di u\|=\|\xi\|\|u\| \qq {\rm and\q similarly} \qq  \|u\di\eta\|=\|\eta\|\|u\|.
\end{align}

Now let $\rr:E\times F\to G$ be a bilinear operator between linear spaces. Its {\it amplification} 
is the bilinear operator $\rr_\ii:HE\times HF\to HG$, associated with the 4-linear operator 
$H\times E\times H\times F\to HG:(\xi,x,\eta,y)\mt(\xi\di\eta)\rr(x,y)$. In other words, $\rr_\ii$ 
is well defined on elementary tensors by $\rr_\ii(\xi x,\eta y)=(\xi\di\eta)\rr(x,y)$ . 
\begin{definition} \label{boundbil}
A bilinear operator $\rr$ between $PL$-spaces is called {\it $L$-bounded}, respectively, {\it 
$L$-contractive}, if its amplification is (just) bounded, respectively, contractive. We put 
$\|\rr\|_{lb}:=\|\rr_\ii\|$. 
\end{definition}
It is easy to see that an $L$-bounded bilinear operator, being considered between respective 
underlying (semi-)normed spaces is just bounded, and $\|\rr\|\le\|\rr\|_{lb}$. On the other hand, 
similarly to linear operators, sometimes the ``classical'' boundedness automatically implies the 
$L$--boundedness. 

\begin{proposition} \label{autcombil}
 Let $E, F$ be $PL$-spaces, $f:E\to\co$, and $g:F\to\co$ bounded functionals.
 Then the bilinear functional $f\times g:E\times F\to\co:(x,y)\mt f(x)g(y)$ is $L$--bounded 
 and $\|f\times g\|_{lb}=\|f\|\|g\|$.
 \end{proposition}

\begin{proof}
Since $\|f\times g\|=\|f\|\|g\|$, it suffices to show that $\|f\times g\|_{lb}\le\|f\|\|g\|$. 
Indeed, combining the obvious formula $(f\times g)_\ii(u,v)=f_\ii(u)\di g_\ii(v)$, Proposition 
\ref{autc} and (\ref{gooddi}), we have $\|(f\times g)_\ii(u,v)\|\le\|f\|\|g\|\|u\|\|v\|$. 
\end{proof}
\begin{proposition} \label{contrbil} 
    Let $L_p(X):=L_p(X,\co)$ and $L_p(X,E)$ be the $PL$-spaces from Example \ref{ex3}. Then the bilinear 
operator $\rr:L_p(X)\times E\to L_p(X,E)$, taking $(z,x)$ to the $E$-valued function $t\mt z(t)x; 
t\in X$, is $L$-contractive. 
\end{proposition}
\begin{proof}
Recall the isometric operator $\al:H(L_p(X,E))\to L_p(X,HE)$ and distinguish its particular case 
$\al_0:H(L_p(X))\to L_p(X,H)$. Also consider the bilinear operator ${\cal S}:L_p(X,H)\times HE\to 
L_p(X,HE)$, taking a pair $(\om,u)$ to the $HE$--valued function $t\mt\om(t)\di u; t\in X$. With 
the help of (\ref{good2}), a routine calculation gives $\|{\cal S}(\om,u)\|=\|\om\|\|u\|$. 

 Now consider the diagram
\[
\xymatrix@C+20pt{H(L_p(X))\times HE \ar[r]^{\rr_\ii}\ar[d]_{\al_0\times\id_{HE}}
& H(L_p(X,E)) \ar[d]^{\al} \\
L_p(X,H)\times HE \ar[r]^{{\cal S}} & L_p(X,HE) },
\]

It is easy to check on elementary tensors in the respective amplifications that it is commutative. 
Therefore, for $w\in H(L_p(X))$ and $u\in HE$ we have 
$$
\|\rr_\ii(w,u)\|=\|\al(\rr_\ii(w,u))\|=\|{\cal S}(\al_0(w),u)\|=\|\al_0(w)\|\|u\|=\|w\|\|u\|. \q
$$
\end{proof}
We shall denote the completion of the $PL$-space $E\ot_p F$ from Example \ref{ex8} by 
$E\widehat\ot_p F$. Clearly, it is the $PL$--quantization of the ``classical'' Banach space 
$E\widehat\ot_p F$. 

\begin{proposition}  \label{pr}
Let $E$ be a normed space, $F$ a $PL$--space, $E\ot_p F$ the resulting $PL$--space. Then the 
canonical bilinear operator $\vartheta:E_{\max}\times F\to E\ot_p F$, considered between the 
respective $PL$--spaces, is $L$--contractive. Moreover, $\widehat\vartheta:E_{\max}\times F\to 
E\widehat\ot_p F$, that is $\vartheta$, considered with $E\widehat\ot_p F$ as its range, is also 
$L$--contractive. 
 \end{proposition}
 \begin{proof}
Consider the trilinear operator ${\cal T}:E\times H\times HF\to E\ot_p(HF):(x,\xi,v)\mt x\ot(\xi\di 
v)$. It follows from (\ref{good2}) that 
${\cal T}$ is contractive. Therefore the bilinear operator 
 ${\cal S}:(E\ot_p H)\times HF\to E\ot_p HF:(x\ot\xi,v)\mt x\ot(\xi\di v)$,
 is also contractive.

 Recall the isometric operator $\beta:H(E\ot_p F)\to E\ot_p HF$ from Example \ref{ex8} and
 distinguish its particular case, the ``flip'' $\beta_0:HE_{\max}\to E\ot_p H$. Consider the diagram
\[
\xymatrix@C+20pt{HE_{\max}\times HF \ar[r]^{\vartheta_\ii}\ar[d]_{\beta_0\times\id_{HF}}
& H(E\ot_p F) \ar[d]^{\beta} \\
(E\ot_p H)\times HF \ar[r]^{{\cal S}} & E\ot_p HF },
\]
\noindent which is obviously commutative. Therefore a routine calculation shows that for $w\in 
HE_{\max}$ and $v\in HF$ we have $\|\vartheta_\ii(w,v)\|\le\|w\|\|u\|$, and we are done. 
\end{proof}
\section{Proto-Lambert tensor product}
We proceed to show that $L$-bounded bilinear operators between $PL$-spaces can be linearized with 
the help of a specific tensor product `` ${\ot_{pl}}$ '', which seems to be new. But before, since 
in this paper we shall come across several varieties of a tensor product, it is convenient to give 
a general definition, embracing all particular cases. 

Let us fix, throughout this section, two {\it arbitrary chosen} $PL$--spaces $E$ and $F$. Further, 
let 
$\mho$ be a subclass of the class of all normed $PL$-spaces. 

\begin{definition} \label{deftp} 
 A pair $(\Theta,\theta)$ that consists of $\Theta\in\mho$ and an
$L$-contractive bilinear operator $\theta:E\times F\to\Theta$ is called {\it tensor product of $E$ 
and $F$ relative to $\mho$} if, for every $G\in\mho$ and every $L$--bounded bilinear operator 
$\rr:E\times F\to G$, there exists a unique $L$--bounded operator $R:\Theta\to G$ such that the 
diagram 
\[
\xymatrix@R-10pt@C+15pt{
E\times F \ar[d]^{\theta} \ar[dr]^{\rr} & \\
\Theta \ar[r]^R  &  G  } 
\]
\noindent is commutative, and moreover $\|R\|_{lb}=\|\rr\|_{lb}$. 
\end{definition}
Such a pair is unique in the following sense: if $(\Theta_k,\theta_k); k=1,2$ are two pairs, 
satisfying the given definition for a certain $\mho$, then there is a $L$--isometric isomorphism 
$I:\Theta_1\to\Theta_2$, such that $I\theta_1=\theta_2$. This fact is a particular case of a 
general--categorical observation concerning the uniqueness of an initial object in a category; cf., 
e.g.,~\cite{mcl},~\cite[Theorem 2.73]{heb3}. However, the question about the existence of such a 
pair depends on our luck with the choice of the class $\mho$. 
%
\begin{definition} \label{deftppl}
The tensor product of $E$ and $F$ relative to the class of all normed $PL$--spaces is called {\it 
non-completed $PL$--tensor product} of our spaces. 
\end{definition}

\medskip
We shall prove the existence of such a pair, displaying its explicit construction. 

First, we need a sort of ``extended''  version of the diamond multiplication, this time between 
elements of amplifications of linear spaces. Namely, for $u\in HE, v\in HF$ we consider the element 
$u\di v:= \vartheta_\ii(u,v)\in H(E\ot F)$, where $\vartheta:E\times F\to E\ot F$ is the canonical 
bilinear operator. In other words,  this ``diamond operation'' is well defined by $\xi x\di\eta 
y:=(\xi\di\eta)(x\ot y).$ 
\begin{proposition} \label{presU} 
Every $U\in H(E\ot F)$ can be represented as $\sum_{k=1}^na_k\cd(u_k\di v_k)$ for some natural $n$ 
and $a_k\in\bb, u_k\in HE,v_k\in HF, k=1,...,n$. 
\end{proposition}
\begin{proof}
Evidently, it suffices to consider the simplest case, when $U=\xi(x\ot y);\xi\in H, x\in E, y\in 
F$. Take arbitrary non-zero $\eta,\zeta\in H$; then we have $\xi=a(\eta\di\zeta)$, for some 
$a\in\bb$. Consequently, $U=a\cd(\eta x\di\zeta y)$. 
\end{proof}
As a corollary, the operator $\bb\ot HE\ot HF\to H(E\ot F)$, associated with the 3-linear operator 
$(a,u,v)\mt a\cd(u\di v)$, is surjective. Thus $H(E\ot F)$ can be endowed with the seminorm of the 
respective quotient space  of $\bb\ot_p HE\ot_p HF$, denoted by $\|\cd\|_{pl}$. 
In other words, we have 
\begin{align} \label{inf}
\|U\|_{pl}:=\inf\{\sum_{k=1}^n\|a_k\|\|u_k\|\|v_k\|\},
\end{align}
where the infimum is taken over all possible representations of $U$ as indicated in Proposition 
\ref{presU}. 
\begin{proposition} \label{contrpl}
The seminormed $\bb$--module $(H(E\ot F),\|\cd\|_{pl})$ is contractive. 
\end{proposition}
\begin{proof}
Clearly, $\bb\ot_p HE\ot_p HF$ is a contractive left $\bb$-module as a tensor product of the left 
$\bb$-module $\bb$ and the linear space $HE\ot HF$. 
 Therefore $H(E\ot F)$ is the image of a contractive left $\bb$-module with respect to a quotient map of 
 seminormed spaces. Since the latter map is a module morphism, we  easily obtain the desired property.
 \end{proof}
Thus, $\|\cd\|_{pl}$ is a $PL$--seminorm on $E\ot F$. Denote the respective $PL$--space by 
$E\ot_{pl}F$. 

 Observe the obvious estimate
\begin{align} \label{dim}
\|u\di v\|_{pl}\le\|u\|\|v\|; \q u\in HE,v\in HF.
\end{align}
Since $u\di v=\vartheta_\ii(u,v)$, we see that {\it $\vartheta$, considered with range 
$E\ot_{pl}F$, is $L$-contractive.} 

Looking at the underlying spaces and using (\ref{gooddi}), we easily obtain that 
\begin{align} \label{underl}
\|x\ot y\|\le\|x\|y\|; \q x\in E, y\in F.
\end{align}
(In fact, in (\ref{dim}) and (\ref{underl}) we have the equality, but we shall not discuss it now). 

%
%
\begin{proposition} \label{comdiag}
Let $G$ be a $PL$-space, $\rr:E\times F\to G$ an $L$--bounded bilinear operator, $R:E\ot_{pl}F\to 
G$ the associated linear operator. Then $R$ is $L$--bounded, and   $\|\rr\|_{lb}\|=\|R\|_{lb}$. 
\end{proposition}
\begin{proof}
 Take $U\in H(E\ot_{pl}F))$ and represent it according to Proposition \ref{presU}. 
 We remember that
$R_\ii$ is a $\bb$-module morphism. Therefore, using the obvious equality $R_\ii(u\di 
v)=\rr_\ii(u,v)$, we have that 
$R_\ii(U)=\sum_{k=1}^na_k\cd\rr_\ii(u_k, v_k)$, hence 
\[
\|R_\ii(U)\|\le\sum_{k=1}^n\|a_k\|\|\rr_\ii(u_k, v_k)\|\le\|\rr\|_{lb}\sum_{k=1}^n\|a_k\|\|u_k\|\|v_k\|.
\]
From this, using \ref{inf}, we  obtain that $\|R_\ii(U)\|\le\|\rr\|_{lb}\|\|U\|_{pl}$. Thus our $R$ 
is $L$--bounded, and $\|R\|_{lb}\le\|\rr\|_{lb}$. The converse inequality easily follows from 
(\ref{dim}). 
\end{proof}
\begin{proposition} \label{norm}
{\rm (As a matter of fact),} $\|\cd\|_{pl}$ is a norm. 
\end{proposition}
\begin{proof}
 By Proposition \ref{pro1}, it is sufficient to show that, for a non-zero elementary tensor $\xi w; 
w\in E\ot_{pl}F, \xi\in H; \|\xi\|=1, w\ne0$ we have $\|\xi w\|_{pl}\ne0$. Since 
  $E$ and $F$ are normed spaces, then, as it is known, there exist bounded functionals 
  $f:E\to\co, g:F\to\co$ such that $(f\ot g)w\ne0$. Now consider in the previous proposition 
  $\rr:=f\times g:E\times F\to\co$. By virtue of Proposition \ref{autcombil}, $\rr$ is $L$--bounded, hence
the operator $(f\ot g)_\ii$ is bounded. At the same time $(f\ot g)_\ii(\xi w)=[(f\ot 
g)(w)]\xi\ne0$, and we are done. 
\end{proof}
Combining Propositions \ref{comdiag} and \ref{norm}, we immediately obtain 
\begin{theorem} \label{exthe}
 {\rm  (Existence theorem).} The pair $(E\ot_{pl}F,\vartheta)$  is a non-completed $PL$-tensor product of $E$ and $F$.
\end{theorem}
\medskip
We can also speak about the ``completed'' version of Definition \ref{deftppl}. 
%
\begin{definition} \label{defbtp}
The tensor product of $E$ and $F$ relative to the class of all complete $PL$--spaces is called {\it 
completed, or Banach $PL$--tensor product} of our spaces. 
\end{definition}
\begin{proposition} \label{exbtp}
The Banach $PL$--tensor product of $PL$--spaces $E$ and $F$ exists, and it is the pair 
$(E\widehat\ot_{pl}F,\widehat\vartheta)$, where $E\widehat\ot_{pl}F $ is the completion of the 
$PL$--space $E\ot_{pl} F$, and $\widehat\vartheta$ acts as $\vartheta$, but with range 
$E\widehat\ot_{pl} F $. 
\end{proposition}
\begin{proof}
This is an immediate corollary of the universal property of the completion. 
\end{proof}
\section{Tensoring by maximal $PL$--spaces and by $L_1(\cd)$}
In this section we show that for certain concrete tensor factors their $PL$--tensor product also 
becomes something concrete and transparent. 
\begin{theorem} \label{th2}
 Let  $E$ be a normed space, $F$ a $PL$--space, $E\ot_p F$ the $PL$--space from Example \ref{ex8}. Then
there exists an $L$--isometric isomorphism $I:E_{\max}\ot_{pl}F\to E\ot_p F$, acting as the 
identity operator on the common underlying linear space of our $PL$--spaces. As a corollary {\rm 
(see Proposition \ref{context}),} there exists an $L$--isometric isomorphism $\widehat 
I:E_{\max}\widehat\ot_{pl}F\to E\widehat\ot_p F$, which is the extension by continuity of $I$. 
\end{theorem}
\begin{proof}
 Consider $\vartheta$ from Proposition \ref{pr}.  By Proposition \ref{exthe}, $\vartheta$ gives rise to the $L$--contractive operator $I$, acting as in the formulation.
 Therefore it is sufficient to show that the operator $I_\ii$ does not decrease norms of elements.

Take $U\in H(E\ot F)$. Identifying the latter space with $E\ot HF$, we can represent $U$ as 
$\sum_{k=1}^n x_k\ot v_k; x_k\in E, v_k\in HF$. Choose $e\in H;\|e\|=1$  and denote by $S\in\bb$ 
the isometric operator $\xi\mt e\di\xi;\xi\in H$. We easily see that 
$$
U=S^*\cd\left[\sum_{k=1}^nex_k\di v_k\right].
$$
From this, by (\ref{inf}), we obtain the estimate $\|U\|_{pl}\le\sum_{k=1}^n\|x_k\|\|v_k\|$. Hence, 
we have 
\begin{align} \label{aa}
\|U\|_{pl}\le\inf\{\sum_{k=1}^n\|x_k\|\|v_k\|\}.
\end{align}
where the infimum is taken over all representations of $U$ in the indicated form. 

Now look at $I_\ii(U)$. It is the same $\sum_{k=1}^n x_k\ot v_k$, only considered in the normed 
space $E\ot_p HF$. It follows that $\|I_\ii(U)\|$ is exactly the infimum, indicated in (\ref{aa}). 
Thus, $\|I_\ii(U)\|\ge\|U\|_{pl}$ and we are done. 
\end{proof}
\begin{remark}
As an easy corollary of this theorem, we have, up to an $L$--isometric isomorphism, that 
$E_{\max}\widehat\ot_{pl} F_{\max}=[E\widehat\ot_{p} F]_{\max}$ for all normed spaces $E$ and $F$. 
In particular, we have $H_{\max}\widehat\ot_{pl}H_{\max}={\cal N}(H)_{\max}$, where ${\cal N}(H)$ 
is the Banach space of trace class operators on $H$. 
\end{remark}
Now we want to apply this theorem to the description of $PL$--tensor products in the situation, 
when one of tensor factors is $L_1(X)$ from Example \ref{ex3}. As a ``classical'' prototype of that 
description, we recall the following theorem, due to Grothendieck. 

{\it Let $(X,\mu)$ be a measure space, $F$ a Banach space. Then there exists an isometric 
isomorphism ${\cal G}_F:L_1(X)\widehat\ot_p F\to L_1(X,F)$, well defined by taking an elementary 
tensor $z\ot x; z\in L_1(X), x\in F$ to the $F$-valued function $t\mt z(t)x; t\in X$.} 
\begin{theorem} \label{grr}
Let $(X,\mu)$ be a measure space, $F$ a complete $PL$--space. Then there exists an $L$--isometric 
isomorphism ${\cal I}:L_1(X)\widehat\ot_{pl} F\to L_1(X,F)$, well defined in the same way as ${\cal 
G}_F$ in the Grothendieck theorem. 
\end{theorem}
\begin{proof}
First we note that {\it the $PL$-norm on $L_1(X)$, introduced in Example \ref{ex3}, coincides with 
the maximal $PL$-norm, 
 introduced in Example \ref{ex1}}.
 This is because the identity operator $I:H\ot_p L_1(X)\to H(L_1(X))$
 participates in the commutative diagram
\[
\xymatrix@C+20pt{H\ot_p L_1(X) \ar[r]^{I}\ar[d]_{flip}
& H(L_1(X)) \ar[d]^{\al_1} \\
L_1(X)\ot_p H \ar[r]^{{\cal G}_0} & L_1(X,H) },
\]
where ${\cal G}_0$ is the restriction of ${\cal G}_H$, $\al_1$ is a particular case of $\al$ from 
Example \ref{ex3}, and these operators, as well as ``flip'', are isometric. 

\medskip
Combining this with Theorem \ref{th2}, we come to the $L$--isometric isomorphism \\ $\widehat I: 
L_1(X)\widehat\ot_{pl} F\to L_1(X)\widehat\ot_p F$. 

\medskip
Now we show  that {\it the isometric isomorphism ${\cal G}_F:L_1(X)\widehat\ot_p F\to L_1(X,F)$ is 
$L$--isometric with respect to the $PL$--norm in the latter space, taken from Example \ref{ex3}.} 
We see that ${\cal G}_F$ is the extension by continuity (cf. Section 3) of its restriction ${\cal 
G}_F^0$ to $L_1(X)\ot_p F$, and this restriction maps the latter space onto a dense subspace of 
$L_1(X,F)$. Therefore, by virtue of Proposition \ref{context}, it is sufficient to show that the 
operator ${\cal G}_F^0$ is $L$--isometric, or, equivalently, that $({\cal G}_F^0)_\ii$ is 
isometric. But the latter participates in the commutative diagram 
\[
\xymatrix@C+20pt{H(L_1(X)\ot_p F) \ar[r]^{({\cal G}_F^0)_\ii}\ar[d]_{\beta}
& H(L_1(X,F)) \ar[d]^{\al} \\
L_1(X)\ot_p HF \ar[r]^{{\cal G}_{HF}^0} & L_1(X,HF) },
\]
\noindent where $\al$ and (with $L_1(X)$ as $E$) $\beta$ are operators from the Examples \ref{ex3} 
and \ref{ex8}, respectively, and ${\cal G}_{HF}^0$ is the restriction to the respective subspaces 
of the isometric isomorphism, provided by the Grothendieck theorem (this time with the completion 
of $HF$ in the role of $F$). Since $\al,\beta$ and ${\cal G}_{HF}^0$ are isometric, then $({\cal 
G}_0)_\ii$ is also isometric. 

\medskip
After this, to end the proof, it remains to set ${\cal I}:={\cal G}_F\widehat I$. 
\end{proof}

Like in the ``classical'' context of Grothendieck theorem, we can distinguish a transparent 
particular case, concerning integrable functions of two variables. 
\begin{proposition} \label{grot}
Let $(X,\mu), (Y,\nu)$ be two measure spaces, $(X\times Y,\mu\times\nu)$ their product measure 
space. Then $L_1(X)\widehat\ot_{pl}L_1(Y)= L_1(X\times Y)$ up to an $L$--isometric isomorphism. 
More precisely, there exists  an $L$--isometric isomorphism between the indicated $PL$--spaces, 
well defined by taking $x\ot y; x\in L_1(X), y\in L_1(Y)$ to the function $(s,t)\mt x(s)y(t); 
(s,t)\in X\times Y$. 
\end{proposition}
\begin{proof}
 By virtue of Theorem \ref{grr}, it is sufficient to show that the ``classical'' isometric isomorphism 
 $I:L_1(X,L_1(Y))\to L_1(X\times Y)$, taking the $L_1(Y)$-valued integrable function $\bar x$ to the 
 function $I(\bar x):(s,t)\mt[\bar x(s)](t)$, is $L$--isometric. Since
every $u\in H(L_1(X,L_1(Y))$ is (finite) sum of elementary tensors, a routine calculation, using 
the construction of the $PL$--norm on relevant $L_p$--spaces and, of course, Fubini Theorem, shows 
that indeed $\|u\|=\|I_\ii(u)\|$. 
\end{proof}
\begin{remark}
The results of this section can lead to the conjecture that something similar, at least in 
formulations, can be said in the context of the so-called proto-quantum spaces in the operator 
space theory (cf.~\cite[Ch. 2]{heb2}). It happened that to some extent it is indeed so, but proofs 
of crucial facts become more complicated. However, we do not discuss it in the present paper. 
\end{remark}
\section{Lambert spaces and the Lambert tensor product}
From now on we concentrate on a special type of $PL$-spaces, which is a non-coordinate form of {\it 
Operatorfolgenr\"aume} of Lambert. 

If $X$ is a left $\bb$-module and $x\in X$, we say that a projection $P\in\bb$ is a {\it support} 
of $x$, if $P\cd x=x$. A {\it contractive} seminormed left $\bb$-module $Y$ is called {\it 
semi--Ruan module}, if it has the following ``property (sR)'': for $x,y\in Y$ with orthogonal 
supports we have 
$$
\|x+y\|^2\le\|x\|^2+\|y\|^2,
$$
 hence for every $x_k\in x; k=1,...,n$ with pairwise orthogonal supports we have $\|\sum_{k=1}^n x_k\|^2\le\sum_{k=1}^n\|x_k\|^2$.

(Semi-Ruan modules were introduced and studied in~\cite{he3}. Then, in more general context and 
with more advanced results, they were investigated in~\cite{wit}. However, earlier the same class 
of modules actually cropped up in~\cite{mag}). 
\begin{definition} \label{lamsp}
For a linear space $E$, a seminorm on $HE$ is called {\it Lambert seminorm,} or briefly, {\it 
$L$--seminorm}, if the left $\bb$-module $HE$ is a semi--Ruan module. In other words, $L$--seminorm 
is a $PL$--seminorm, satisfying the property (sR). The linear space, endowed with an $L$--seminorm, 
is called {\it seminormed Lambert space}, or briefly {\it seminormed $L$--space}. In a similar way, 
we use the term {\it normed Lambert space} ( = {\it normed $L$--space)}, but in this case we 
usually omit the word ``normed''. 
\end{definition}
As to examples of $PL$--spaces, it is easy to see that $E_{\min}$ is actually an $L$--space for all 
normed spaces $E$, whereas $E_{\max}$ is, generally speaking, not an $L$--space. The $PL$--space 
$L_p(X)$ is an $L$--space if, and only if $2\le p$. (Of course, we suppose here that our measure 
space is not a single atom). Finally, if $K$ is a pre--Hilbert space, we can endow it with a 
so-called {\it Hilbert $L$-norm}, obtained after identifying $HK$ with $H\ot_{hil}K$. 

The following example will not be used in this paper. However, we mention it because of its 
importance  in the theory of $L$--spaces. 
\begin{example} \label{ex4}
   (``Concrete $L$--space''). Suppose that $E$ is given as a subspace of $\bb(K,L)$ for some Hilbert spaces $K,L$. Consider the operator $\gamma:HE\to\bb(K, H\hil L)$,
well defined by taking $\xi T: \xi\in H, T\in E$ to the operator $x\mt\xi\ot T(x); x\in K$. 
Introduce a seminorm on $HE$, setting $\|u\|:=\|\gamma(u)\|$. Then it is not difficult to show that 
$\|\cd\|$ is actually a norm, making $E$ an $L$--space. 

As a matter of fact, this example is, in a sense, universal: every $L$--space is $L$--isometrically 
isomorphic to some concrete, i.e. operator space. This assertion can be rather quickly derived from 
the non--coordinate version of the result of Lambert~\cite[Folgerung 1.3.6]{lam} about the 
embedding of his spaces into products of copies of $\ell_2$. However, details are outside the scope 
of this paper. 
\end{example}

\medskip
From now on we proceed to a special tensor product within the class of $L$--spaces. As we shall 
see, its definition is parallel to that of the $PL$--tensor product, but the resulting object turns 
out to be a quite different thing. 

Let us fix, for a time, two $PL$-spaces $E$ and $F$. 
\begin{definition}
 The tensor product of $E$ and $F$ relative to
the class of all normed $L$--spaces is called {\it non-completed $L$--tensor product of our 
spaces.} 
\end{definition}
\begin{remark} \label{maxtp}
 This tensor product is a ``non--coordinate' version of what Lambert calls maximal tensor product.
 Indeed, in a sense, it is maximal within a reasonable class of tensor products, and it plays in Lambert's theory
 a role similar to the role of the operator-projective tensor product in the theory of
 quantum ( = abstract operator) spaces. See details in~\cite[3.1.1]{lam}.
\end{remark}
We shall prove the existence of this kind of tensor product, displaying its explicit construction. 
Such a construction and the crucial Proposition \ref{triang} can be considered as the 
``non--coordinate'' version of what was done by Lambert. 

Recall the diamond product in all its varieties. The following proposition concerns all linear 
spaces $E,F$ without any additional structure. First, note two identities 
\begin{align} \label{usf}
a\cd u\di b\cd v=(a\di b)(u\di v);\,\,\, 
a\cd[(b\cd u)\di v]=a(b\di\id)\cd(u\di v); a,b\in\bb, u\in HE, v\in HF
\end{align}
 that can be easily checked on elementary tensors. 
\begin{proposition} \label{repr3}
    Every $U\in H(E\ot F)$ can be represented as $a\cd(\sum_{k=1}^n u_k\di v_k)$, 
$a\in\bb, u_k\in HE,v_k\in HF, k=1,...,n$, where 
$u_k$ have pairwise orthogonal supports. 
\end{proposition}
\begin{proof}
Represent $U$ as in Proposition \ref{presU}. Choose isometric operators $S_1,...,S_n\in\bb$ with 
pairwise orthogonal final projections. Set 
$$
a:=\sum_{k=1}^na_kS^*_k\di\id \qq {\rm and} \qq u_k':=S_k\cd u_k; k=1,...,n.
$$
Since $S^*_kS_l=\de^l_k\id$, the identities (\ref{usf}) imply that 
$$
a\cd\left[\sum_{k=1}^nu'_k\di v_k\right]=\sum_{k,l=1}^n[a_k(S^*_k\di\id)]\cd[(S_l\cd u_l)\di v_l]=
$$
$$
\sum_{k,l=1}^n[a_k(S^*_k\di\id)(S_l\di\id)]\cd[u_l\di v_l]=\sum_{k=1}^na_k\cd
(u_k\di v_k)=U.
$$
Finally, the elements $u'_k$ have orthogonal supports $S_kS^*_k; k=1,...,n$. 
\end{proof}

\medskip
Now, for a given $U\in H(E\ot F)$, we introduce the number 
 \begin{align}  \label{inf2}
\|U\|_l=\inf\left\{\|a\|\left(\sum_{k=1}^n\|u\|^2\|v\|^2\right)^{\frac{1}{2}}\right\},
\end{align}
where the infimum is taken over all possible representations of $U$ in the form indicated by 
Proposition \ref{repr3}. We distinguish the obvious 
\begin{proposition} \label{con}
For every $U\in H(E\ot F)$ and $a\in\bb$ we have $\|a\cd U\|_l\le\|a\|\|U\|_l$. 
\end{proposition}
\begin{proposition} \label{triang}
 The function $U\mt\|U\|_l$ is a seminorm on $H(E\ot F)$.
\end{proposition}
\begin{proof}
 Let $U=a\cd(\sum_{k=1}^n u^1_k\di v^1_k)$, $V=b\cd(\sum_{l=1}^mu^2_l\di v^2_l)$, where the elements
$u^1_k\in HE$, respectively $u^2_l\in HE$, have pairwise orthogonal supports $P_k$, respectively 
$Q_l$. Take arbitrary isometric operators $S,T\in\bb$ with orthogonal final projections and observe 
that 
$$
U+V=(a(S^*\di\id)+b(T^*\di\id))\cd\left(\sum_{k=1}^n (S\cd u^1_k)\di v^1_k+\sum_{l=1}^m(T\cd u^2_l)\di
v^2_l\right).
$$

The elements $S\cd u^1_k$ have supports $SP_kS^*$, whereas the elements $T\cd u^2_l$ have supports 
$TQ_lT^*$; hence, taking together, these elements have pairwise orthogonal supports. Therefore, by 
virtue of Proposition \ref{con} and of (\ref{inf2}), we have that 
$$
\|U+V\|_l\le\|a(S^*\di\id)+b(T^*\di\id)\|
\left(\sum_{k=1}^n\|(S\cd u^1_k)\|^2\|v^1_k\|^2+\sum_{l=1}^m\|(T\cd u^2_l\|^2\|v^2_l\|^2\right)^{\frac{1}{2}}.
$$
Combining this with the operator  $C^*$-property, we obtain that 
$$
\|U+V\|_l\le(\|a\|^2+\|b\|^2)^{\frac{1}{2}}\left(\sum_{k=1}^n\|
u^1_k)\|^2\|v^1_k\|^2+\sum_{l=1}^m\|u^2_l\|^2\|v^2_l\|^2\right)^{\frac{1}{2}}.
$$
Further, obviously we can assume that 
$$
\|a\|=\left(\sum_{k=1}^n\|u_k^1\|^2\|v^2_k\|^2\right)^\frac{1}{2} \q {\rm 
and} \q \|b\|=\left(\sum_{l=1}^m\|u_l^2\|^2\|v^2_l\|^2\right)^\frac{1}{2}.
$$
Therefore we have 
$$
\|U+V\|_l\le\|a\|^2+\|b\|^2=\|a\|\left(\sum_{k=1}^n\|u^1_k\|^2\|v^1_k\|^2\right)^{\frac{1}{2}}+
\|b\|\left(\sum_{l=1}^m\|u^2_l\|^2\|v^2_l\|^2\right)^{\frac{1}{2}}.
$$
From this with the help of (\ref{inf2}) we obtain that $\|U+V\|_l\le\|U\|_l+\|V\|_l$. 

The property of seminorms, concerning the scalar multiplication, is immediate. 
\end{proof}
\begin{proposition} \label{sR}
 The module $(H(E\ot F),\|\cd\|_l)$ has the property (sR).
\end{proposition}
\begin{proof}
 Let $U,V\in H(E\ot F)$ have orthogonal supports $P$ and $Q$. Choose their arbitrary suitable representation, say
$U=a\cd(\sum_{k=1}^n u^1_k\di v^1_k)$, $V=b\cd(\sum_{l=1}^mu^2_l\di v^2_l)$. Take $S,T$ as in 
Proposition \ref{triang}; then, by a similar argument, we have 
$$
\|U+V\|_l\le\|a(S^*\di\id)+b(T^*\di\id)\|\left(\sum_{k=1}^n\|u^1_k\|^2\|v^1_k\|^2+
\sum_{l=1}^m\|u^2_l\|^2\|v^2_l\|^2\right)^{\frac{1}{2}}.
$$
Evidently, we can assume that $a=Pa,b=Qb$ and $\|a\|=\|b\|=1$. Therefore, by the operator 
$C^*$-property, we have 
$$
\|a(S^*\di\id)+b(T^*\di\id)\|=\|Paa^*P+Qbb^*Q\|^{\frac{1}{2}}=\max\{\|a\|,\|b\|\}=1.
$$
Consequently, by (\ref{inf2}), we have that $\|U+V\|_l^2\le\|U\|_l^2+\|V\|_l^2$. 
\end{proof}
Combining the last three propositions, we see that {\it $\|\cd\|_l$ is an $L$--seminorm on $E\ot 
F$.} 
 We denote the resulting semi-normed $L$-space by $E\ot_l F$.

Like in the ``$PL$--case'' (cf. (\ref{dim})), we have the obvious estimation 
\begin{align} \label{dim2}
\|u\di v\|_l\le\|u\|\|v\|; \q u\in HE,v\in HF.
\end{align}
Consequently, {\it the canonical bilinear operator $\vartheta:E\times F\to E\ot_l F$ is 
$L$-contractive.} 

%
%
\begin{proposition} \label{comd2}
 Let $G$ be an $L$-space, $\rr:E\times F\to G$ an $L$--bounded bilinear
operator, $R:E\ot_{pl}F\to G$ the associated linear operator. Then $R$ is $L$-bounded, and   
$\|\rr\|_{lb}=\|R\|_{lb}$. 
\end{proposition}
\begin{proof}
Take $U\in H(E\ot_{l}F))$ and represent it as in Proposition \ref{repr3}. Since $R_\ii$ is a 
$\bb$-module morphism, and $R_\ii(u_k\di v_k)=\rr_\ii(u_k,v_k)$ for all $k$, we have
that $R_\ii(U)=a\cd(\sum_{k=1}^n\rr_\ii(u_k, v_k))$. 

Now look at $\rr_\ii(u_k,v_k)\in HG$ for some $k$. 
Obviously we have 
 $(P_k\di\id)\cd\rr_\ii(u_k,v_k)=\rr_\ii(P_k\cd u_k, v_k)$. This implies that the
elements $\rr_\ii(u_k,v_k)\in HG$ have pairwise orthogonal supports, namely $P_k\di\id$. Therefore, 
since $G$ is an $L$-space, we have 
$$
\|R_\ii(U)\|\le\|a\|\left(\sum_{k=1}^n\|\rr_\ii(u_k,v_k)\|^2\right)^\frac{1}{2}\le
$$
$$
\|a\|\left(\sum_{k=1}^n\|\rr\|_{lb}^2\|u_k\|^2\|v_k\|^2\right)^\frac{1}{2}=
\|\rr\|_{lb}\|a\|\left(\sum_{k=1}^n\|u_k\|^2\|v_k\|^2\right)^\frac{1}{2}.
$$

 From this, using (\ref{inf2}),
we obtain the estimate $\|R_\ii(U)\|\le\|\rr\|_{lb}\|\|U\|_l\|$. Consequently, 
$\|R\|_{lb}\le\|\rr\|_{lb}$. The converse inequality easily follows from (\ref{dim2}). 
\end{proof}
\begin{proposition} \label{norm4}
{\rm (As a matter of fact),} $\|\cd\|_{l}$ is a norm. 
\end{proposition}
\begin{proof}
Since $\co$ is an $L$--space, the argument in Proposition \ref{norm} works with obvious 
modifications. 
\end{proof}
Combining Propositions \ref{con}--\ref{norm4}, we immediately obtain 
\begin{theorem} \label{exlam}
   {\rm (Existence theorem)} The pair $(E\ot_{l}F,\vartheta)$  is a non-completed $L$--tensor product of $E$ and $F$.
\end{theorem}
\medskip
The non-completed $L$--tensor product has an obvious ``completed'' version. The definition of the 
{\it completed $L$--tensor product of two $PL$--spaces} and the relevant existence theorem repeat 
what was said about completed $PL$--tensor product, only we replace ``$PL$'' by ``$L$'' and the 
subscript  ``$pl$'' by ``$l$''. Thus, {\it the completed $L$--tensor product of two $PL$--spaces 
$E$ and $F$ exists, and it is the pair $(E\widehat\ot_{l}F,\widehat\vartheta)$, where 
$E\widehat\ot_{p}F $ is the completion of the $L$--space $E\ot_{l} F$, and $\widehat\vartheta$ acts 
as $\vartheta$, but with range $E\widehat\ot_{l} F $.} 

\begin{remark} \label{none}
We do not discuss here the non-coordinate version of another tensor product, the so--called 
minimal, 
  introduced in~\cite[3.1.3]{lam}. Unlike the maximal tensor product (cf.
 Remark \ref{maxtp}), it corresponds, in a sense, to the
 operator--injective tensor product in the operator space theory.
\end{remark}
\section{The Lambert tensor product of Hilbert spaces}
As we have seen before, the $PL$--tensor product is especially good for maximal $PL$--spaces and 
$L_1$-spaces with their specific $PL$--norm. Here we shall show that, in the same sense, the 
$L$--tensor product is good for Hilbert spaces with 
the minimal $L$--norm, that is $\|\cd\|_{\min}$ of Example \ref{ex1}. {\it Throughout this section, 
all Hilbert and pre--Hilbert spaces are supposed to be endowed with that $L$--norm}. 

We shall use one of equivalent definitions of the minimal $L$--norm. It is a particular case of the  
definition of the norm in the injective tensor product $E\ot_i F$ of two normed spaces, expressed 
by means of an injective operator $E\ot_i F\to\bb(E^*,F)$ (see, e.g.,\cite[pp. 62-63]{clm}). In the 
particular case of a pre-Hilbert space $K$ we obtain the following observation that we distinguish 
for the convenience of references. 
\begin{proposition} \label{isom}
There is an isometric operator ${\cal I}:HK\to{\bb}({K}^{cc},H)$, 
  well defined by $\xi x\mt\xi\circ x$.
 \end{proposition}
This, in its turn, implies 
\begin{proposition} \label{orth}
{\rm (i)} For $u=\sum_{k=1}^n\lm_k\xi_kx_k\in HK$, where $\lm_k\in\co$ and $\xi_k,x_k$ are 
orthonormal systems in $H$ and $K$ respectively, we have $\|u\|=\max\{|\lm_k|; k=1,..., n\}$. 

{\rm (ii)} Every $u\in HK$ can be represented as $\sum_{k=1}^ns_k\xi_k x_k$, where $\xi_k$ and 
$x_k$ are orthonormal systems in $H$ and $K$ respectively, $s_1\ge s_2\ge ...\ge s_n>0$. 
\end{proposition}
\begin{proof}
  (i) is immediate. To prove (ii), we recall that ${\cal I}(u)$, being a {\it finite rank} operator
 between pre--Hilbert spaces, has the form $\sum_{k=1}^ns_k\xi_k\circ x_k$, where
$\xi_k, x_k$ and $s_k$ have the indicated properties. (E.g., the argument in~\cite[Section 
3.4]{heb3} works with obvious modifications). It follows that $u$ have the desired representation. 
 \end{proof}
 %
%
\begin{proposition} \label{lcon}
    Let $K$ and $L$ be pre-Hilbert spaces. Then the canonical bilinear operator 
    $\vartheta:K\times L\to K\ot_{hil} L: (x,y)\mt x\ot y$ is $L$-contractive.
\end{proposition}
\begin{proof}
 As we know, $\vartheta_\ii:HK\times HL\to  H(K\ot_{hil} L)$
takes a pair $(u,v)$ to $u\di v$. By Proposition \ref{orth}(ii), $u$ has the form 
$\sum_{k=1}^ns_k\xi_kx_k$ with the mentioned properties, and $v$ has the form 
$\sum_{l=1}^ms_l'\eta_ly_l$ with similar properties. Consequently, 
$$
u\di v=\sum_{k=1}^n\sum_{l=1}^ms_ks_l'(\xi_k\di\eta_l)(x_k\ot y_l),
$$
where the systems $\xi_k\di\eta_l$ and $x_k\ot y_l$ are orthonormal in $H$ and $K\ot_{hil} L$, 
respectively. Therefore, by Proposition \ref{orth}(i), $\|u\di v\|=s_1s_1'=\|u\|\|v\|$. 
\end{proof}
\begin{theorem} \label{theo4} 
Let $K$ and $L$ be pre--Hilbert spaces. Then we have $K\ot_l L=K\ot_{hil} L$ and $K\widehat\ot_l 
L=K\widehat\ot_{hil} L$. Both equalities are up to an $L$--isometric  isomorphism, well defined by 
taking an elementary tensor $x\ot y$ to the same $x\ot y$, but considered in $K\ot_{hil} L$ and 
$K\widehat\ot_{hil} L$, respectively. 
\end{theorem}
\begin{proof}
Since $\vartheta$ from the previous proposition has values in an $L$--space, it gives rise 
 to the $L$--contractive operator $R: K\ot_{l} L\to K\ot_{hil} L$, which is the identity map of the 
 underlying linear spaces. Our task is to show that it is an $L$--isometric isomorphism.

Take $U\in H(K\ot_{l}L)$. Since it is the sum of several elementary tensors of the form $\xi(x\ot 
y)$, it easily follows that $U$ can be represented as $\sum_{k=1}^n\sum_{l=1}^{m}\xi_{kl}(x_{k}\ot 
y_{l})$, where $x_k$ and $y_l$ are orthonormal systems in $K$ and $L$, respectively, and 
$\xi_{kl}\in H$. Applying to $R_\ii(U)\in H(K\ot_{hil} L)$ Proposition \ref{isom}, we see that 
$\|R_\ii(U)\|$ is the norm of the operator $S:=\sum_{k=1}^n\sum_{l=1}^{m}\xi_{kl}\circ(x_k\ot 
y_l):(K\ot_{hil} L)^{cc}\to H$. Set $M:=span\{x_k\ot y_l\}\subset(K\ot_{hil} L)^{cc}$. Since $\dim 
M<\ii$, the pre--Hilbert space $(K\mmdd L)^{cc}$ decomposes as $M\oplus M^\perp$, and $S$ takes 
$M^\perp$ to 0. Therefore $\|S\|=\|S_0\|$, where $S_0$ is the restriction of $S$ to $M$. 
Thus, we have 
\begin{align} \label{end1}
\|R_\ii(U)\|=\|S_0\|.
\end{align}
 Now return to our initial $U$. Choose arbitrary orthonormal systems $\eta_k; k=1,...,n$ and 
$\zeta_l;l=1,...,m$ in $H$. Set $u:=\sum_{k=1}^n\eta_kx_k\in HK$ and 
$v:=\sum_{l=1}^{m}\zeta_ly_l\in HL$. We see that $u\di 
v=\sum_{k=1}^n\sum_{l=1}^{m}(\eta_k\di\zeta_l)(x_k\ot y_l)$. 
Consider the finite rank operator 
$$
T:=\sum_{k=1}^n\sum_{l=1}^{m}\xi_{kl}\circ(\eta_k\di\zeta_l):H\to H.
$$
An easy calculation shows that $T\cd(u\di v)=U$. Further, as a particular case of Proposition 
\ref{isom}, $\|u\|=\|v\|=1$. Finally, if we set $N:=span\{\eta_k\di\zeta_l\}$ and denote by $T_0$ 
the restriction of $T$ to $N$, we obviously have $\|T\|=\|T_0\|$. Combining this with (\ref{inf2}), 
we obtain that 
\begin{align} \label{end2}
\|U\|_l\le\|T_0\|.
\end{align}
The systems $\{x_k\ot y_l\}$ and $\{\eta_k\di\zeta_l\}$ are orthonormal bases in $M$ and $N$, 
respectively, and $S_0(x_k\ot y_l)=\xi_{kl}=T_0(\eta_k\di\zeta_l)$. It follows that 
$\|T_0\|=\|S_0\|$. Combining this with (\ref{end1}) and (\ref{end2}), and remembering that $R_\ii$ 
is contractive, we obtain that $\|U\|_l=\|R_\ii(U)\|$. 

 This gives the first of $L$--isometric isomorphisms, claimed in the theorem. The extension of the latter
by continuity provides the second $L$--isometric isomorphism. 
\end{proof}
\section{Comparison of both tensor products}
In conclusion, we want to compare $PL$-- and $L$--tensor products. Since the class of $PL$-spaces 
is larger than that of $L$-spaces, it immediately follows from the definition of both tensor 
products in terms of their universal properties, that $\|\cd\|_{pl}\ge\|\cd\|_{l}$. We shall show 
that the first number is sometimes essentially greater than the second number. 

Endow the space $\ell_2$ with the Hilbert $L$--norm (see above), and the space $\ell_1=L_1({\mathbb 
N})$ with the $PL$-norm from Example \ref{ex3}. 
\begin{proposition} \label{con6}
    The bilinear operator ${\cal M}:\ell_2\times \ell_2\to \ell_1$, acting as the coordinate-wise multiplication, is $L$-contractive.
\end{proposition}
\begin{proof}
Our task is to show that the bilinear operator ${\cal M}_\ii:H\ell_2\times H\ell_2\to H\ell_1$ is 
contractive. Consider the isometric operators $I:H\ell_2=H\ot_{hil} \ell_2\to \ell_2(H)$ and 
$\al:H\ell_1\to \ell_1(H)$; both are well defined by taking an elementary tensor 
$\xi\widetilde\lm;\widetilde\lm=(\dots,\lm_n,\dots)$ to $(\dots,\lm_n\xi,\dots)$. Consider the 
diagram 
\[
\xymatrix@C+20pt{H\ell_2\times H\ell_2 \ar[r]^{I\times I}\ar[d]_{{\cal M}_\ii}
& \ell_2(H)\times \ell_2(H) \ar[d]^{{\cal S}} \\
H\ell_1 \ar[r]^{\al} & \ell_1(H) },
\]
\noindent where ${\cal S}$ takes a pair 
$(\widetilde\xi:=(\dots,\xi_n,\dots),\widetilde\eta:=(\dots,\eta_n,\dots))$ to 
$(\dots,\xi_n\di\eta_n,\dots)$. Since the Cauchy-Schwarz inequality implies that $\|{\cal 
S}(\widetilde\xi,\widetilde\eta)\|\le \|\widetilde\xi\|\|\widetilde\eta\|$, the latter sequence 
indeed belongs to $\ell_1(H)$, and, moreover, ${\cal S}$ is contractive. 

Now observe that our diagram, as one can easily verify on elementary tensors, is commutative. 
Consequently, for $u,v\in H\ell_2$ we have $\|{\cal M}_\ii(u,v)\|\le\|u\|\|v\|$. 
\end{proof}
Have a look at the $PL$-- and $L$--tensor square of the same Hilbert $L$--space $\ell_2$. Fix 
$n\in{\Bbb N}$, denote by ${\bf p}^m; m=1,2,...$ 
sequences $(\dots,0,1,0,\dots)\ell_2$ and choose an arbitrary orthonormal system, say 
$e_1,e_2,\dots$, in $H$. In what follows, we set 
\[ 
V:=\sum_{k=1}^ne_k({\bf p}^k\ot{\bf p}^k)\in H(\ell_2\ot 
\ell_2). 
\]
It obviously can be presented as 
\begin{align} \label{presV}
V=S\cd\sum_{k=1}^nu_k\di v_k,
\end{align}
where $u_k=v_k=e_k{\bf p}^k$, and $S=\sum_{k=1}^ne_k\circ(e_k\di e_k)$.
%
\begin{proposition} \label{n}
 We have $\|V\|_{pl}=n$.
\end{proposition}
\begin{proof}
Consider the operator $M:\ell_2\ot_{pl} \ell_2\to \ell_1$, associated with ${\cal M}$ from the 
previous proposition. Then it is $L$--contractive together with the latter: in particular, 
$\|M_\ii(V)\|\le\|V\|_{pl}$. But we obviously have $M_\ii(V)=\sum_{k=1}^ne_k({\bf p}^k)$; 
therefore, since we are in $L\ell_1$, we have $\|M_\ii(V)\|=n$. Consequently, $\|V\|_{pl}\ge n$. On 
the other hand, it follows from (\ref{presV}) that 
$\|V\|_{pl}\le\sum_{k=1}^n\|S\|\|u_k\|\|v_k\|=n$. 
\end{proof}
\begin{proposition} \label {sqr}
 {\rm (At the same time)} we have $\|V\|_l=\sqrt{n}$.
\end{proposition}
\begin{proof}
Consider the bilinear operator ${\cal N}:\ell_2\times \ell_2\to \ell_2$, acting as ${\cal M}$, but 
with the other range. Since the norm of an element in 
 $H\ell_1$ can only decrease, if we shall consider this element in $H\ell_2$, our  ${\cal N}$ is 
 $L$--contractive together with ${\cal M}$. But $\ell_2$ (contrary to $\ell_1$!) is an $L$--space; 
 therefore the operator
$N:\ell_2\ot_{l} \ell_2\to \ell_2$, associated with ${\cal N}$, is also $L$--contractive. In 
particular, $\|N_\ii(V)\|\le\|V\|_{l}$. Of course, $N_\ii(V)$ is the same $\sum_{k=1}^ne_k({\bf 
p}^k)$ as in the previous proposition. However, since now we are in $H\ell_2= \\ H\ot_{hil} 
\ell_2$, we have $\|N_\ii(V)\|=\sqrt{n}$, and hence $\|V\|_{l}\ge\sqrt{n}$. 

On the other hand, since the elements $u_k\in H\ell_2$ form an orthonormal system, we obtain, by 
(\ref{presV}), that $\|V\|_{l}\le\|S\|(\sum_{k=1}^n\|u_k\|^2\|v_k\|^2)^{\frac{1}{2}}=\sqrt{n}$. 
\end{proof}

Nevertheless, despite $PL$-- and $L$--tensor products of the same $PL$-spaces usually have 
essentially different norms, their underlying spaces coincide: 
\begin{proposition}
Let $E$ and $F$ be $PL$--spaces. Then the identity operator on the linear space $E\ot F$ is an 
isometric isomorphism, being considered as an operator between underlying spaces of $PL$--spaces 
$E\ot_{pl}F$ and $E\ot_l F$. 
\end{proposition}
\begin{proof}
Since the latter operator is obviously contractive, our task is to show that its inverse operator 
is also contractive. Denote by $G$ be the underlying normed space of $E\ot_{pl}F$, endowed with the 
minimal $PL$-norm (see Example \ref{ex1}). Since $\vartheta:E\times F\to E\ot_{pl}F$ is 
$L$--contractive, the same is true, if we consider $\vartheta$ with range $G$. But, as we know, $G$ 
is an $L$--space. Therefore $\vartheta$ gives rise to the $L$--contractive operator between $E\ot_l 
F$ and $G$, which is contractive as an operator between the underlying normed spaces. But the 
latter is, of course, the desired inverse operator.
 \end{proof}
\bigskip
This research was supported by the Russian Foundation for Basic Research (grant No. 15-01-08392).

\begin{flushleft}
Moscow State (Lomonosov) University\\ Moscow, 111991, Russia\\
E-mail address: helemskii@rambler.ru 
\end{flushleft}

\ed